\documentclass[a4paper,12pt]{article}

\usepackage{amsfonts,graphicx,subfig,amsmath,amssymb,enumitem,
url,color,authblk,amsthm,multicol,cite,float,longtable,caption,
rotating}

\DeclareMathOperator\erf{erf}

\theoremstyle{definition}

\theoremstyle{theorem}
\newtheorem{lemma}{Lemma}[section]
\theoremstyle{theorem}
\newtheorem{theorem}{Theorem}[section]
\theoremstyle{corollary}

\theoremstyle{remark}

\title{{\sc {\Large Determination of one unknown thermal coefficient
through a mushy zone model with a convective overspecified boundary condition}}}

\author{Ceretani, Andrea N. and Tarzia, Domingo A.\\{\small CONICET - Depto. Matem\'atica, Facultad de Ciencias Empresariales, Universidad Austral,\\
Paraguay 1950, S2000FZF Rosario, Argentina.\\ E-mail: \textcolor{blue}{aceretani@austral.edu.ar}; \textcolor{blue}{dtarzia@austral.edu.ar}}}

\date{}

\begin{document}
\maketitle

\begin{abstract}
A semi-infinite material under a solidification process with the Solomon-Wilson- Alexiades's mushy
zone model with a heat flux condition at the fixed boundary is considered. The associated
free boundary problem is overspecified through a convective boundary condition with the
aim of the simultaneous determination of the temperature, the two free boundaries of the
mushy zone and one thermal coefficient among the latent heat by unit mass, the thermal
conductivity, the mass density, the specific heat and the two coefficients that characterize
the mushy zone. Bulk temperature and coefficients which characterize the heat flux and
the heat transfer at the boundary are assumed to be determined experimentally.
Explicit formulae for the unknowns are given for the resulting six phase-change problems, beside necessary and sufficient conditions
on data in order to obtain them. In addition, relationship between the phase-change
process solved in this paper with an analogous process overspecified by a
temperature boundary condition is presented, and this second problem is solved
by considering a large heat transfer coefficient at the boundary in the problem with the convective boundary condition.
Formulae for the unknown thermal coefficients corresponding
to both problems are summarized in two tables.
\end{abstract}


{\bf 2010 AMS subjet classification}: 35R35 - 35C06 - 80A22

\section{Introduction}\label{sec:Intro}
Heat transfer problems with a phase-change such as melting and freezing have been studied in the last century due to their wide scientific and technological applications. Some books in the subject are \cite{AlSo1993, Ca1984, CaJa1959, Cr1984, Fa2005, Gu2003, Hi1987, Lu1991, Ta1986}.

In this paper we consider a phase-change process for a semi-infinite material, which is characterized by $x>0$, that is initially assumed to be liquid at its melting temperature (which without loss of generality we assume equal to $0^\circ$C). We consider this material under a solidification process with the presence of a zone where solid and liquid coexist, known as "mushy zone", with a heat flux boundary condition imposed at the fixed face $x=0$. We follow \cite{SoWiAl1982, Ta1987} in considering three different regions in this type of solidification process:
\begin{enumerate}
\item liquid region at temperature $T(x,t)=0$:
$D_l=\left\{(x,t)\in\mathbb{R}^2/\,x>r(t),\,t>0\right\}$,
\item solid region at temperature $T(x,t)<0$:
$D_s=\left\{(x,t)\in\mathbb{R}^2/\,0<x<s(t),\,t>0\right\}$,
\item mushy region at temperature $T(x,t)=0$:
$D_p=\left\{(x,t)\in\mathbb{R}^2/\,s(t)<x<r(t),\,t>0\right\}$,
\end{enumerate}
being $x=s(t)$ and $x=r(t)$ the functions that characterize the free boundaries of the mushy zone. We also follow \cite{SoWiAl1982} in making the following assumptions on the structure of the mushy zone, which is considered as isothermal:
\begin{enumerate}
\item the material contains a fixed portion of the total latent heat per unit mass (see condition (\ref{3}) below),
\item its width is inversely proportional to the gradient of temperature (see condition (\ref{4}) below).
\end{enumerate}

Parameters involved in this problem are:\\

\noindent \hspace*{2cm}$l>0$: latent heat by unit mass,\\
\hspace*{2cm}$k>0$: thermal conductivity,\\
\hspace*{2cm}$\rho>0$: mass density,\\
\hspace*{2cm}$c>0$: specific heat,\\
\hspace*{2cm}$0<\epsilon<1$: one of the two coefficients which characterize the mushy zone,\\
\hspace*{2cm}$\gamma>0$: one of the two coefficients which characterize the mushy zone,\\
\hspace*{2cm}$q_0>0$: coefficient that characterizes the heat flux at $x=0$,\\
\hspace*{2cm}$h_0>0$: coefficient that characterizes the heat transfer at $x=0$,\\
\hspace*{2cm}$-D_\infty<0$: bulk temperature at $x=0$.\\

We suppose that five of the six thermal coefficients $l$, $k$, $\rho$, $c$, $\epsilon$ and $\gamma$ of the solid phase are known and that, by means of a change of phase experiment (solidification of the material at its melting temperature) we are able to measure the quantities
$q_0$, $h_0$ and $-D_\infty$.

Encouraged by the recent works \cite{TaInPress,TaToAppear} and with the aim of the simultaneous determination of the temperature $T=T(x,t)$, the two free boundaries $x=r(t)$ and $x=s(t)$, and one unknown thermal coefficient among $l$, $k$, $\rho$, $c$, $\epsilon$ and $\gamma$, we impose an overspecified boundary condition \cite{Ca1984} which consists of the specification of a convective condition at the fixed face $x=0$ (see condition (\ref{7}) below) of the material undergoing the phase-change process. This lead us to the following free boundary problem:
\begin{align}
\label{1}&\rho c T_t(x,t)-kT_{xx}(x,t)=0&0<x<s(t),\hspace{0.25cm}&t>0\\
\label{2}&T(s(t),t)=0& &t>0\\
\label{3}&kT_x(s(t),t)=\rho l[\epsilon\dot{s}(t)+(1-\epsilon)\dot{r}(t)]& &t>0\\
\label{4}&T_x(s(t),t)(r(t)-s(t))=\gamma& &t>0\\
\label{5}&r(0)=s(0)=0\\
\label{6}&kT_x(0,t)=\frac{q_0}{\sqrt{t}}& &t>0\\
\label{7}&kT_x(0,t)=\frac{h_0}{\sqrt{t}}(T(0,t)+D_\infty)& &t>0
\end{align}

This problem was first study in \cite{Ta1987} with a temperature boundary condition at $x=0$ instead of the convective condition (\ref{7}) considered in this paper. Moreover, the determination of one unknown thermal
coefficient for the one-phase Lam\'e-Clapeyron-Stefan problem with an overspecified heat flux condition at the fixed face $x=0$ without a mushy zone was done in \cite{Ta1982}. Other papers related to determination of thermal coefficients are \cite{BaZaFr2010,BoMh2012,DaMiUp2009,ErLeHa2013,Go2013,HaPe2013,
HaIsLeKe2013,HeNoSlWiZi2013,Hr2012,HuLe2014,HuLeIv2014,InOnNaKu2007,
KaIs2012,KeIs2012,LaZh1995,LiYaWaJi2011,LuHaJiZhWuLi2015,
PeZhLiYaGu2014,SaTa2007,SaTa2011,VaDa2009,WaLiRe2006,YaZh2006,
YaYuDe2008,ZhWuJiRh2014,ZuLiJiGu2009}.

The goal of this paper is to obtain the explicit solution to phase-change process (\ref{1})-(\ref{7}) with one unknown thermal coefficient, and the necessary and sufficient conditions on data in order to obtain an explicit formula for the unknown thermal coefficient.
In addition, we are interested in analysing the relationship between problem (\ref{1})-(\ref{7}) and the phase-change process given by (\ref{1})-(\ref{6}) beside the Dirichlet boundary condition overspecified at $x=0$ given by (\ref{8}) (see below). In particular, we are interested in solving the problem with Dirichlet boundary condition through problem with convective boundary condition when large values of the coefficient $h_0$ that characterizes the heat transfer at $x=0$ are considered.

The organization of the paper is as follows. In Section \ref{sec:ExpSol} we prove a preliminary result where necessary and sufficient conditions on data for the phase-change process (\ref{1})-(\ref{7}) are given in order to obtain the temperature $T=T(x,t)$ and the two free boundaries $x=r(t)$ and $x=s(t)$. Based on this preliminary result, in Section \ref{sec:ExpForm} we present and solve six different cases for the phase-change process (\ref{1})-(\ref{7}) according the choice of the unknown thermal coefficient among $l$, $k$, $\rho$, $c$, $\epsilon$ and $\gamma$.
In Section \ref{sec:largeh0} we discuss the relationship between the phase-change process (\ref{1})-(\ref{6}) with the Dirichlet boundary condition (\ref{8}) and the same process with the convective boundary condition (\ref{7}). We show that the temperature $T_D=T_D(x,t)$, the free boundaries $x=r_D(t)$ and $x=s_D(t)$, and the explicit formula for the unknown thermal coefficient $l$, $k$, $\rho$, $c$, $\epsilon$ or $\gamma$ for the phase-change process (\ref{1})-(\ref{6}) with the Dirichlet condition (\ref{8}) can be obtained through the phase-change process with convective condition given by (\ref{1})-(\ref{7}) when $h_0$ tends to $+\infty$. Explicit formulae for the unknown thermal coefficient for problems (\ref{1})-(\ref{7}) and (\ref{1})-(\ref{6}) and (\ref{8}), beside restrictions on data that guarantees their validity, are summarizes in Table 1 and 2 respectively.

\section{Explicit solution to the phase-change process}\label{sec:ExpSol}

The following lemma represents the base on which the work in this Section will be structured.

\begin{lemma}\label{th:0}
The solution to problem (\ref{1})-(\ref{7}) is given by:
\begin{align}
\label{sT}&T(x,t)=\frac{q_0\sqrt{\pi\alpha}}{k}\left[\erf\left(\frac{x}{2\sqrt{\alpha t}}\right)-\erf(\xi)\right]&0<x<s(t),\,\,&t>0\\
\label{ss}&s(t)=2\xi\sqrt{\alpha t}& &t>0\\
\label{sr}&r(t)=2\left[\xi+\frac{\gamma k}{2q_0\sqrt{\alpha}}\exp{(\xi^2)}\right]\sqrt{\alpha t}& &t>0
\end{align}
if and only if the parameters involved in problem (\ref{1})-(\ref{7})
satisfy the following two equations:
\begin{align}
\label{eq:1}&\left[\xi+\frac{\gamma k(1-\epsilon)}{2q_0\sqrt{\alpha}}\exp{(\xi^2)}\right]\exp{(\xi^2)}=\frac{q_0}{\rho l\sqrt{\alpha}}\\
\label{eq:2}&\erf(\xi)=\frac{k D_\infty}{q_0\sqrt{\pi\alpha}}\left(1-\frac{q_0}{h_0D_\infty}\right)
\end{align}
where $\alpha=\frac{k}{\rho c}$ represents the thermal diffusivity.
\end{lemma}

\begin{proof}
The kind of phase-change processes considered in this article have the following general solution \cite{SoWiAl1982,Ta1987,TaToAppear}:
\begin{align}
\label{s1}&T(x,t)=A+B\erf\left(\frac{x}{2\sqrt{\alpha t}}\right)& 0<x<s(t),\,&t>0\\
\label{s2}&s(t)=2\xi\sqrt{\alpha t}& &t>0\\
\label{s3}&r(t)=2\mu\sqrt{\alpha t}& &t>0
\end{align}
where coefficients $A$, $B$, $\xi$ and $\mu$ depend on the particular specifications of the phase-change process.

In orden to have the solution to problem (\ref{1})-(\ref{7}), we impose conditions (\ref{2})-(\ref{4}), (\ref{6}) and (\ref{7}) on (\ref{s1})-(\ref{s3}) and obtain that coefficients $A$, $B$ and $\mu$ must be given by:
\begin{equation}{\label{coef}}
A=-\frac{q_0\sqrt{\pi\alpha}}{k}\erf(\xi),\hspace{2cm}
B=\frac{q_0\sqrt{\pi\alpha}}{k}\hspace{1cm}
\text{and}\hspace{1cm}
\mu=\xi+\frac{\gamma k\exp{(\xi^2)}}{2q_0\sqrt{\alpha}}\text{,}
\end{equation}
which corresponds to solution (\ref{sT})-(\ref{sr}),
and the parameters involved in the problem must satisfy equations (\ref{eq:1}) and (\ref{eq:2}).
\end{proof}

As a consequence of Lemma \ref{th:0}, we know that we can solve the phase-change process (\ref{1})-(\ref{7}) with one unknown thermal coefficient through the determination of the parameter $\xi$ that characterizes one of the two free boundaries of the mushy zone and the unknown thermal coefficient among $l$, $k$, $\rho$, $c$, $\epsilon$ and $\gamma$. In addition, we also know from Lemma \ref{th:0} that we can do that by solving the system of equations (\ref{eq:1})-(\ref{eq:2}).

\section{Explicit formula for the unknown thermal coefficient}\label{sec:ExpForm}

In this Section we present and solve six different cases for the phase-change process (\ref{1})-(\ref{7}) according the choice of the unknown thermal coefficient among $l$, $k$, $\rho$, $c$, $\epsilon$ and $\gamma$.

With the aim of organizing our work, we classify each case by making reference to the coefficients which is necessary to know in order to solve it (see Lemma \ref{th:0}):\\

Case 1: Determination of $l$ and $\xi$,\hspace{2cm}
Case 2: Determination of $\gamma$ and $\xi$,\\

Case 3: Determination of $\epsilon$ and $\xi$,\hspace{2cm}
Case 4: Determination of $k$ and $\xi$,\\

Case 5: Determination of $\rho$ and $\xi$,\hspace{2cm}
Case 6: Determination of $c$ and $\xi$.\\

\noindent In addition, with the goal of make our presentation more readable, in the following statements and proofs we introduce several functions. We name this functions with a subscript according the case where they arise.

\begin{theorem}[Case 1: determination of $l$ and $\xi$]\label{th:1}
If in problem (\ref{1})-(\ref{7}) we consider the thermal parameter
$l$ as an unknown, then its solution is given by
(\ref{sT})-(\ref{sr}) with $l$ and $\xi$ given by:
\begin{align}
\label{l1}&l=\sqrt{\frac{c}{\rho k}}\frac{q_0\exp{(-\xi^2)}}{\left[\xi+\frac{\gamma(1-\epsilon)\sqrt{k\rho c}}{2q_0}\exp{(\xi^2)}\right]}\\
\label{xi1}&\xi=\erf^{-1}\left(\frac{D_\infty}{q_0}\sqrt{\frac{k\rho c}{\pi}}\left(1-\frac{q_0}{h_0D_\infty}\right)\right)
\end{align}
if and only if the parameters $q_0$, $h_0$, $D_\infty$, $k$, $\rho$ and $c$ satisfy the following two inequalities:
\begin{align}
&\label{R1} 1-\frac{q_0}{h_0D_\infty}>0\tag{R1}\\
&\label{R2} \frac{D_\infty}{q_0}\sqrt{\frac{k\rho c}{\pi}}\left(1-\frac{q_0}{h_0D_\infty}\right)<1\tag{R2}\text{.}
\end{align}
\end{theorem}

\begin{proof}
Due to properties of the error function, it follows that a necessary and sufficient condition for the existence and uniqueness of a positive solution to equation (\ref{eq:2}) is:
\begin{equation*}
0<\frac{D_\infty}{q_0}\sqrt{\frac{k\rho c}{\pi}}\left(1-\frac{q_0}{h_0D_\infty}\right)<1\text{,}
\end{equation*}
which is equivalent to inequalities (\ref{R1}) and (\ref{R2}). In that case, the positive solution to equation (\ref{eq:2}) is given by (\ref{xi1}). Finally, it follows from equation (\ref{eq:1}) that $l$ is the positive thermal coefficient given by (\ref{l1}).
\end{proof}

\begin{theorem}[Case 2: determination of $\gamma$ and $\xi$]\label{th:2}
If in problem (\ref{1})-(\ref{7}) we consider the thermal parameter
$\gamma$ as an unknown, then its solution is given by
(\ref{sT})-(\ref{sr}) with $\gamma$ given by:
\begin{equation}\label{gamma2}
\gamma=\frac{2q_0}{(1-\epsilon)\sqrt{k\rho c}}\left(\frac{q_0}{l}\sqrt{\frac{c}{\rho k}}-f_2(\xi)\right)\exp{(-2\xi^2)}
\end{equation}
and $\xi$ given by (\ref{xi1}),
if and only if the parameters $q_0$, $h_0$, $D_\infty$, $k$, $\rho$, $c$ and $l$ satisfy inequalities (\ref{R1}), (\ref{R2}) and:
\begin{equation}\label{R3}
f_2\left(\erf^{-1}\left(\frac{D_\infty}{q_0}\sqrt{\frac{k\rho c}{\pi}}\left(1-\frac{q_0}{h_0D_\infty}\right)\right)\right)<
\frac{q_0}{l}\sqrt{\frac{c}{\rho k}}\tag{R3}\text{,}
\end{equation}
where the real function $f_2$ is defined by:
\begin{equation}\label{f2}
f_2(x)=x\exp{(x^2)},\hspace{0.5cm}x>0\text{.}
\end{equation}
\end{theorem}

\begin{proof}
As we see in the proof of Theorem \ref{th:1}, a necessary and sufficient condition that guarantees the existence and uniqueness of solution to equation (\ref{eq:2}) is that inequalities (\ref{R1}) and (\ref{R2}) hold, and in that case, the coefficient $\xi$ is given by (\ref{xi1}).

On the other hand, it follows from equation (\ref{eq:1}) that $\gamma$ is given by (\ref{gamma2}). This coefficient is positive if and only if:
\begin{equation}\label{ineq}
f_2(\xi)<\frac{q_0}{l}\sqrt{\frac{c}{\rho k}}\text{,}
\end{equation}
where $f_2$ is the real function defined in (\ref{f2}).
Taking into account the expression of $\xi$ given in (\ref{xi1}), we have that inequality (\ref{ineq}) is equivalent to inequality (\ref{R3}).
\end{proof}

\begin{theorem}[Case 3: determination of $\epsilon$ and $\xi$]\label{th:3}
If in problem (\ref{1})-(\ref{7}) we consider the thermal parameter
$\epsilon$ as an unknown, then its solution is given by
(\ref{sT})-(\ref{sr}) with $\epsilon$ given by:
\begin{equation}\label{epsilon3}
\epsilon=1-\frac{2q_0}{\gamma \sqrt{k\rho c}}\left(\frac{q_0}{l}\sqrt{\frac{c}{\rho k}}-f_2(\xi)\right)\exp{(-2\xi^2)}
\end{equation}
and $\xi$ given by (\ref{xi1}),
if and only if the parameters $q_0$, $h_0$, $D_\infty$, $k$, $\rho$, $c$ and $\gamma$ satisfy inequalities (\ref{R1}), (\ref{R2}), (\ref{R3}) and:
\begin{equation}
\begin{split}\label{R4}
\frac{\gamma \sqrt{k\rho c}}{2q_0}\exp{\left(2\left[\erf^{-1}\left(\frac{D_\infty}{q_0}\sqrt{\frac{k\rho c}{\pi}}\left(1-\frac{q_0}{h_0D_\infty}\right)\right)\right]^2\right)}+\\
f_2\left(\erf^{-1}\left(\frac{D_\infty}{q_0}\sqrt{\frac{k\rho c}{\pi}}\left(1-\frac{q_0}{h_0D_\infty}\right)\right)\right)>\frac{q_0}{l} \sqrt{\frac{c}{\rho k}}\text{,}
\end{split}\tag{R4}
\end{equation}
where $f_2$ is the real function defined in (\ref{f2}).
\end{theorem}

\begin{proof}
Conditions (\ref{R1}) and (\ref{R2}), and the expression of $\xi$ given in (\ref{xi1}) arise in the same way that in the precedent proofs. On the other hand, it follows from equation (\ref{eq:1}) that $\epsilon$ is given by (\ref{epsilon3}), being $f_2$ the real function defined in (\ref{f2}). This coefficient is positive if and only if:
\begin{equation}\label{ineqq}
\frac{2q_0}{\gamma \sqrt{k\rho c}}\left(\frac{q_0}{l}\sqrt{\frac{c}{\rho k}}-f_2(\xi)\right)\exp{(-2\xi^2)}<1\text{.}
\end{equation}
Taking into account the expression of $\xi$ given in (\ref{xi1}), we have that inequality (\ref{ineqq}) is equivalent to inequality (\ref{R4}). Finally, we have that $\epsilon$ given in (\ref{epsilon3}) is less than 1 if and only if $f_2(\xi)<\frac{q_0}{\rho l\sqrt{\alpha}}$, which, as we see in the proof of Theorem \ref{th:2}, is equivalent to condition (\ref{R3}).
\end{proof}

\begin{theorem}[Case 4: determination of $k$ and $\xi$]\label{th:4}
If in problem (\ref{1})-(\ref{7}) we consider the thermal parameter
$k$ as an unknown, then its solution is given by
(\ref{sT})-(\ref{sr}) with $k$ given by:
\begin{equation}\label{k4}
k=\frac{\pi}{\rho c}\left[\frac{q_0\erf(\xi)}{D_\infty\left(1-\frac{q_0}{h_0D_\infty}\right)}\right]^2
\end{equation}
and $\xi$ the unique solution of the equation:
\begin{equation}\label{eq:xi4}
f_4(x)=\frac{cD_\infty}{l\sqrt{\pi}}\left(1-\frac{q_0}{h_0D_\infty}\right),\hspace{0.5cm}x>0\text{,}
\end{equation}
where $f_4$ is the real function defined by:
\begin{equation}\label{f4}
f_4(x)=\left[x+\frac{\gamma\sqrt{\pi}(1-\epsilon)}{2D_\infty\left(1-\frac{q_0}{h_0D_\infty}\right)}\erf(x)\exp{(x^2)}\right]\erf(x)\exp{(x^2)},\hspace{0.5cm}x>0\text{,}
\end{equation}
if and only if the parameters $q_0$, $h_0$ and $D_\infty$ satisfy inequality (\ref{R1}).
\end{theorem}

\begin{proof}
The system of equations (\ref{eq:1})-(\ref{eq:2}) is equivalent to:
\begin{align}
&\sqrt{k}=\sqrt{\frac{\pi}{\rho c}}\frac{q_0\erf(\xi)}{D_\infty\left(1-\frac{q_0}{h_0D_\infty}\right)}\\
&f_4(\xi)=\frac{cD_\infty}{l\sqrt{\pi}}\left(1-\frac{q_0}{h_0D_\infty}\right)\text{,}
\end{align}
where $f_4$ is the real function defined in (\ref{f4}). A necessary condition for existence of solution to this system is that inequality (\ref{R1}) holds. Then, if we assume that (\ref{R1}) holds, we inmediately obtain that $k$ is given by (\ref{k4}). To complete the proof only remains to demonstrate that equation (\ref{eq:xi4}) admits a unique positive solution. This follows from the fact that $f_4$ is an increasing function such that $f(0^+)=0$ and $f(+\infty)=+\infty$.
\end{proof}

\begin{theorem}[Case 5: determination of $\rho$ and $\xi$]\label{th:5}
If in problem (\ref{1})-(\ref{7}) we consider the thermal parameter
$\rho$ as an unknown, then its solution is given by
(\ref{sT})-(\ref{sr}) with $\rho$ given by:
\begin{equation}\label{rho5}
\rho=\frac{\pi}{k c}\left[\frac{q_0\erf(\xi)}{D_\infty\left(1-\frac{q_0}{h_0D_\infty}\right)}\right]^2
\end{equation}
and $\xi$ the unique solution of the equation (\ref{eq:xi4}),
if and only if the parameters $q_0$, $h_0$ and $D_\infty$ satisfy inequality (\ref{R1}).
\end{theorem}

\begin{proof}
It is similar to the proof of Theorem \ref{th:4}.
\end{proof}

\begin{theorem}[Case 6: determination of $c$ and $\xi$]\label{th:6}
If in problem (\ref{1})-(\ref{7}) we consider the thermal parameter
$c$ as an unknown, then its solution is given by
(\ref{sT})-(\ref{sr}) with $c$ given by:
\begin{equation}\label{c6}
c=\frac{\pi}{\rho k}\left[\frac{q_0\erf(\xi)}{D_\infty\left(1-\frac{q_0}{h_0D_\infty}\right)}\right]^2
\end{equation}
and $\xi$ the unique solution of the equation:
\begin{equation}\label{eq:xi6}
f_6(x)=\frac{q_0^2\sqrt{\pi}}{\rho lkD_\infty\left(1-\frac{q_0}{h_0D_\infty}\right)},\hspace{0.5cm}x>0\text{,}
\end{equation}
where $f_6$ is the real function defined by:
\begin{equation}\label{f6}
f_6(x)=\left[\frac{x}{\erf(x)}+\frac{\gamma\sqrt{\pi}(1-\epsilon)}{2D_\infty\left(1-\frac{q_0}{h_0D_\infty}\right)}\exp{(x^2)}\right]\exp{(x^2)},\hspace{0.5cm}x>0\text{,}
\end{equation}
if and only if the parameters $q_0$, $h_0$ and $D_\infty$ satisfy inequalities (\ref{R1}) and:
\begin{equation}\label{R5}
1-\frac{q_0}{h_0D_\infty}<\frac{1}{D_\infty}\left[\frac{2q_0^2}{\rho lk}-\gamma(1-\epsilon)\right]\tag{R5}\text{.}
\end{equation}
\end{theorem}

\begin{proof}
The system of equations (\ref{eq:1})-(\ref{eq:2}) is equivalent to:
\begin{align}
&\sqrt{c}=\sqrt{\frac{\pi}{\rho k}}\frac{q_0\erf(\xi)}{D_\infty\left(1-\frac{q_0}{h_0D_\infty}\right)}\\
&f_6(x)=\frac{q_0^2\sqrt{\pi}}{\rho lkD_\infty\left(1-\frac{q_0}{h_0D_\infty}\right)}
\end{align}
where $f_6$ is the real function defined in (\ref{f6}). A necessary condition for existence of solution to this system is that inequality (\ref{R1}) holds. Then, if we assume that (\ref{R1}) holds, we immediately obtain that $c$ is given by (\ref{c6}). To complete the proof only remains to demonstrate that equation (\ref{eq:xi6}) admits a unique positive solution. Since $f_6$ is an increasing function such that $f(0^+)=\frac{\pi}{2}+\frac{\gamma(1-\epsilon)\sqrt{\pi}}{2D_\infty\left(1-\frac{q_0}{h_0D_\infty}\right)}$ and $f(+\infty)=+\infty$, it follows that a necessary and sufficient condition for existence (and uniqueness) of solution to equation (\ref{eq:xi6}) is that:
\begin{equation*}
\frac{\pi}{2}+\frac{\gamma(1-\epsilon)\sqrt{\pi}}{2D_\infty\left(1-\frac{q_0}{h_0D_\infty}\right)}<
\frac{q_0^2\sqrt{\pi}}{\rho lkD_\infty\left(1-\frac{q_0}{h_0D_\infty}\right)}\text{,}
\end{equation*}
which is equivalent to inequality (\ref{R5}).
\end{proof}

Table 1 summarizes the results of this section, corresponding to 6 cases.

\begin{sidewaystable}[ph!]
\begin{tabular}{|c|c|c|c|}
\hline
& & & \\
Case&Thermal coefficient&Coefficient $\xi$ that characterizes& Restrictions on data\\
&& the free boundary $x=s(t)$& \\
& & & \\
\hline
& & & \\
1&$l=\sqrt{\frac{c}{\rho k}}\frac{q_0\exp{(-\xi^2)}}{\left[\xi+\frac{\gamma(1-\epsilon)\sqrt{k\rho c}}{2q_0}\exp{(\xi^2)}\right]}$ &
$\xi=\erf^{-1}\left(\frac{D_\infty}{q_0}\sqrt{\frac{k\rho c}{\pi}}\left(1-\frac{q_0}{h_0D_\infty}\right)\right)$ & (\ref{R1}), (\ref{R2})\\
& & & \\
\hline
& & & \\
2&\hspace*{0.5cm}$\gamma=\frac{2q_0}{(1-\epsilon)\sqrt{k\rho c}}\left(\frac{q_0}{l}\sqrt{\frac{c}{\rho k}}-\xi\exp{(\xi^2)}\right)\exp{(-2\xi^2)}$\hspace{0.5cm} &
$\xi=\erf^{-1}\left(\frac{D_\infty}{q_0}\sqrt{\frac{k\rho c}{\pi}}\left(1-\frac{q_0}{h_0D_\infty}\right)\right)$ & (\ref{R1}), (\ref{R2}), (\ref{R3})\\
& & &\\
\hline
& & & \\
3&$\epsilon=1-\frac{2q_0}{\gamma \sqrt{k\rho c}}\left(\frac{q_0}{l}\sqrt{\frac{c}{\rho k}}-\xi\exp{(\xi^2)}\right)\exp{(-2\xi^2)}$ &
$\xi=\erf^{-1}\left(\frac{D_\infty}{q_0}\sqrt{\frac{k\rho c}{\pi}}\left(1-\frac{q_0}{h_0D_\infty}\right)\right)$ & (\ref{R1}), (\ref{R2}), (\ref{R3}), (\ref{R4})\\
& & &\\
\hline
& & & \\
& & Unique positive solution of:& \\
4&$k=\frac{\pi}{\rho c}\left[\frac{q_0\erf(\xi)}{D_\infty\left(1-\frac{q_0}{h_0D_\infty}\right)}\right]^2$&$f_4(x)=\frac{cD_\infty}{l\sqrt{\pi}}\left(1-\frac{q_0}{h_0D_\infty}\right)$,
with $f_4$ defined by (\ref{f4})&(\ref{R1})\\
& & &\\
\hline
& & & \\
& & Unique positive solution of:& \\
5&$\rho=\frac{\pi}{kc}\left[\frac{q_0\erf(\xi)}{D_\infty\left(1-\frac{q_0}{h_0D_\infty}\right)}\right]^2$&$f_4(x)=\frac{cD_\infty}{l\sqrt{\pi}}\left(1-\frac{q_0}{h_0D_\infty}\right)$,
with $f_4$ defined by (\ref{f4})&(\ref{R1})\\
& & &\\
\hline
& & & \\
& & Unique positive solution of:&\\
6&$c=\frac{\pi}{\rho k}\left[\frac{q_0\erf(\xi)}{D_\infty\left(1-\frac{q_0}{h_0D_\infty}\right)}\right]^2$&$f_6(x)=\frac{q_0^2\sqrt{\pi}}{\rho lkD_\infty\left(1-\frac{q_0}{h_0D_\infty}\right)}$,
with $f_6$ defined by (\ref{f6})&(\ref{R1}),
(\ref{R5})\\
& & &\\
\hline
\end{tabular}
\caption{{\sc Formulae for problem (\ref{1})-(\ref{7})}. Explicit formulae for the unknown thermal coefficient $l$, $\gamma$, $\epsilon$, $k$, $\rho$ or $c$ and coefficient $\xi$\\ (or the equation that it must satisfy) and the corresponding restrictions on data that guarantees their validity.}
\end{sidewaystable}

\section{The phase-change process with large heat transfer coefficient}\label{sec:largeh0}

A similar phase-change process to (\ref{1})-(\ref{7}) with one unknown thermal coefficient have been studied in \cite{Ta1987}. In that paper, the author consider a fusion process with a mushy zone given by (\ref{1})-(\ref{6}) overspecified with a temperature boundary condition and obtain the explicit solution to some cases. Encouraged by \cite{Ta1987}, let us consider the solidification process (\ref{1})-(\ref{6}) with one unknown thermal coefficient overspecified with the Dirichlet boundary condition:
\begin{equation}\label{8}
T(0,t)=-D_\infty,\hspace{0.5cm}t>0\text{.}
\end{equation}
We can see this condition as the limit case of the convective boundary condition (\ref{7}) when the heat transfer coefficient $h_0$ tends to $+\infty$. From a physical point of view, overspecify the phase-change process (\ref{1})-(\ref{6}) by imposing the convective boundary condition (\ref{7}) seems to be more appropriate than imposing the Dirichlet condition (\ref{8}). This Section is devoted to show that the temperature $T_D=T_D(x,t)$, the free boundaries $x=r_D(t)$ and $x=s_D(t)$, and the explicit formula for the unknown thermal coefficient $l$, $k$, $\rho$, $c$, $\epsilon$ or $\gamma$ for the phase-change process with Dirichlet boundary condition given by (\ref{1})-(\ref{6}) and (\ref{8}) can be obtained through the phase-change process with convective boundary condition given by (\ref{1})-(\ref{7}) when $h_0$ tends to $+\infty$.

We begin with a result related to the solution to the phase-change process (\ref{1})-(\ref{6}) and (\ref{8}). This result may be shown in much the same manner as Lemma \ref{th:0}, thus we do not give its proof here.
\begin{lemma}\label{th:00}
The solution to problem (\ref{1})-(\ref{6}) and (\ref{8}) is given by (\ref{sT})-(\ref{sr}), that is:
\begin{align*}
&T_D(x,t)=\frac{q_0\sqrt{\pi\alpha}}{k}\left[\erf\left(\frac{x}{2\sqrt{\alpha t}}\right)-\erf(\xi)\right]&0<x<s_D(t),\,\,&t>0\\
&s_D(t)=2\xi\sqrt{\alpha t}& &t>0\\
&r_D(t)=2\left[\xi+\frac{\gamma k}{2q_0\sqrt{\alpha}}\exp{(\xi^2)}\right]\sqrt{\alpha t}& &t>0\text{,}
\end{align*}
where $\alpha=\frac{k}{\rho c}$ represents the thermal diffusivity, if and only if the parameters involved in problem (\ref{1})-(\ref{6}) and (\ref{8}) satisfy the following two equations:
\begin{align}
\label{eq:11}&\left[\xi+\frac{\gamma (1-\epsilon)\sqrt{\pi}}{2D_\infty}\erf{(\xi)}\exp{(\xi^2)}\right]\erf{(\xi)}\exp{(\xi^2)}=\frac{c D_\infty }{l\sqrt{\pi}}\\
\label{eq:21}&\erf(\xi)=\frac{D_\infty}{q_0}\sqrt{\frac{k\rho c}{\pi}}\text{.}
\end{align}
\end{lemma}

From Lemma \ref{th:00} we know that, in order to have the temperature $T_D=T_D(x,t)$, the free boundaries $x=r_D(t)$ and
$x=s_D(t)$, and the unknown thermal coefficient $l$, $k$, $\rho$, $c$, $\epsilon$ or $\gamma$ for problem
(\ref{1})-(\ref{6}) and (\ref{8}), it is enough to find the unknown thermal coefficient and the parameter that
characterizes the free boundary $s_D(t)$. Proceeding analogously to the work done in \cite{Ta1987} or in Section
\ref{sec:ExpForm}, we can obtain the thermal coefficient $l$, $k$, $\rho$, $c$, $\epsilon$ or $\gamma$ and the
parameter $\xi$ for problem (\ref{1})-(\ref{6}) and (\ref{8}).
Formulae for those quantities, besides restrictions on data that guarantee their validity, are summarized in Table 2 (restrictions on data and definitions on functions mentioned in Table 2 are listed below the table).

\begin{sidewaystable}[hp!]
\begin{tabular}{|c|c|c|c|}
\hline
& & & \\
Case&Thermal coefficient&Coefficient $\xi$ that characterizes& Restrictions on data\\
& & the free boundary $x=s_D(t)$& \\
& & & \\
\hline
& & & \\
1&$l=\sqrt{\frac{c}{\rho k}}\frac{q_0\exp{(-\xi^2)}}{\left[\xi+\frac{\gamma(1-\epsilon)\sqrt{k\rho c}}{2q_0}\exp{(\xi^2)}\right]}$ &
$\xi=\erf^{-1}\left(\frac{D_\infty}{q_0}\sqrt{\frac{k\rho c}{\pi}}\right)$ & (\ref{R6})\\
& & & \\
\hline
& & & \\
2&\hspace*{0.5cm}$\gamma=\frac{2q_0}{(1-\epsilon)\sqrt{k\rho c}}\left(\frac{q_0}{l}\sqrt{\frac{c}{\rho k}}-\xi\exp{(\xi^2)}\right)\exp{(-2\xi^2)}$\hspace{0.5cm} &
$\xi=\erf^{-1}\left(\frac{D_\infty}{q_0}\sqrt{\frac{k\rho c}{\pi}}\right)$ & (\ref{R7})\\
& & & \\
\hline
& & & \\
3&$\epsilon=1-\frac{2q_0}{\gamma \sqrt{k\rho c}}\left(\frac{q_0}{l}\sqrt{\frac{c}{\rho k}}-\xi\exp{(\xi^2)}\right)\exp{(-2\xi^2)}$ &
$\xi=\erf^{-1}\left(\frac{D_\infty}{q_0}\sqrt{\frac{k\rho c}{\pi}}\right)$ & (\ref{R7}), (\ref{R8})\\
& & & \\
\hline
& & & \\
& & Unique positive solution of:& \\
4&$k=\frac{\pi}{\rho c}\left[\frac{q_0\erf(\xi)}{D_\infty}\right]^2$&$F_4(x)=\frac{cD_\infty}{l\sqrt{\pi}}$,
with $F_4$ defined by (\ref{F4})&\\
& & &\\
\hline
& & & \\
& & Unique positive solution of:& \\
5&$\rho=\frac{\pi}{kc}\left[\frac{q_0\erf(\xi)}{D_\infty}\right]^2$&$F_4(x)=\frac{cD_\infty}{l\sqrt{\pi}}$,
with $F_4$ defined by (\ref{F4})&\\
& & & \\
\hline
& & & \\
& & Unique positive solution of:&\\
6&$c=\frac{\pi}{\rho k}\left[\frac{q_0\erf(\xi)}{D_\infty}\right]^2$&$F_6(x)=\frac{q_0^2\sqrt{\pi}}{\rho lkD_\infty}$,
with $F_6$ defined by (\ref{F6})&(\ref{R9})\\
& & &\\
\hline
\end{tabular}
\caption{{\sc Formulae for problem (\ref{1})-(\ref{6}) and (\ref{8})}. Explicit formulae for the unknown thermal coefficient $l$, $\gamma$, $\epsilon$, $k$, $\rho$ or $c$\\ and coefficient $\xi$ (or the equation that it must satisfy) and the corresponding restrictions on data that guarantees their validity.}
\end{sidewaystable}

\hspace*{0.4cm} List of restrictions on data for problem (\ref{1})-(\ref{6}) and (\ref{8}) mentioned in Table 2:
\begin{equation}\label{R6}
\frac{D_\infty}{q_0}\sqrt{\frac{k\rho c}{\pi}}<1\tag{R6}\hspace{9.6cm}
\end{equation}
\begin{equation}\label{R7}
\frac{D_\infty}{q_0}\sqrt{\frac{k\rho c}{\pi}}<\erf{(\eta)}\tag{R7}\hspace{8.6cm}
\end{equation}
\hspace*{2.5cm}where $\eta$ is the unique positive solution to the equation:
\begin{equation}
f_2(x)=\frac{q_0}{l}\sqrt{\frac{c}{\rho k}},\hspace{0.5cm}x>0\text{,}
\end{equation}
\hspace*{2.5cm} being $f_2$ the real function defined in (\ref{f2}).
\begin{equation}\label{R8}
\frac{D_\infty}{q_0}\sqrt{\frac{k\rho c}{\pi}}>\erf{(\eta)},\hspace{0.5cm}x>0\text{,}\tag{R8}\hspace{6.8cm}
\end{equation}
\hspace*{2.5cm} where $\eta$ is the unique positive solution to the equation:
\begin{equation}
f_2(x)+\frac{\gamma\sqrt{k\rho c}}{2q_0}\exp{(2x^2)}=\frac{q_0}{l}\sqrt{\frac{c}{k\rho}},\hspace{0.5cm}x>0\text{,}
\end{equation}
\hspace*{2.5cm} being $f_2$ the real function defined in (\ref{f2}).
\begin{equation}\label{R9}
\frac{lk\rho D_\infty}{2q_0}\left(1+\frac{\gamma(1-\epsilon)}{D_0}\right)<1\tag{R9}\hspace{6.7cm}
\end{equation}

\hspace*{0.4cm} List of definitions of functions related to problem (\ref{1})-(\ref{6}) and (\ref{8}) mentioned in Table 2:
\begin{equation}\label{F4}
F_4(x)=\erf{(x)}f_2(x)+\frac{(1-\epsilon)\gamma\sqrt{\pi}}{2D_\infty}\left[\erf{(x)}\right]^2\exp{(2x^2)},\hspace{0.5cm}x>0\text{,}\hspace{0.5cm}
\end{equation}
\hspace*{2.5cm} being $f_2$ the real function defined in (\ref{f2}).
\begin{equation}\label{F6}
F_6(x)=\left(\frac{x}{\erf{(x)}}+\frac{\gamma(1-\epsilon)\sqrt{\pi}}{2D_\infty}\exp{(x^2)}\right)\exp{(x^2)},\hspace{0.5cm}x>0\hspace{1cm}
\end{equation}

On the other hand, it is not difficult to verify that formulae and restrictions on data given in Table 2 correspond to
formulae and restrictions on data given in Table 1 for $h_0$ tending to $+\infty$. This fact, besides Lemmas \ref{th:0}
and \ref{th:00}, allow us to conclude that we can solve the phase-change process (\ref{1})-(\ref{6}) with one
unknown thermal coefficient overspecified by the Dirichlet condition (\ref{8}) through the
phase-change process (\ref{1})-(\ref{7}), which is overspecified by the more physically appropriate convective
boundary condition (\ref{7}), when the heat transfer coefficient $h_0$ tends to $+\infty$.

\section{Conclusions}
In this paper, we consider a semi-infinite material under a solidification process with a mushy zone caused by an initial heat flux boundary condition. We solve the associated free boundary problem overspecified with a convective boundary condition and obtain the temperature, the two free boundaries of the mushy zone and one thermal coefficient among the latent heat by unit mass, the thermal conductivity, the mass density, the specific heat and the two coefficients that characterize the mushy zone, when the bulk temperature and the coefficients that characterize the heat flux and the heat transfer at the boundary are assumed to be known. As a consequence, we give formulae for the temperature, the two free boundaries and the unknown thermal coefficient, beside necessary and sufficient conditions on data in order to obtain them.
In addition, we present the relationship between the phase-change process studied in this paper with another similar phase-change process which is overspecified by a Dirichelt boundary condition. From this relationship, we solve the problem with the Dirichlet condition by considering a large heat transfer coefficient in the problem with the convective condition. In this way, we solve the phase-change process overspecified with a temperature boundary condition through the more physically appropriate phase-change problem overspecified with a convective boundary condition. We summarize explicit formulae for the unknown thermal coefficient for both problems in Tables 1 and 2.

\section*{Competing interests}
The authors declare that there is no conflict of interests regarding the publication of this paper.

\section*{Acknowledgements}
This paper has been partially sponsored by the Project PIP No. 0534 from CONICET-UA (Rosario, Argentina) and AFOSR-SOARD Grant FA
9550-14-1-0122.

\bibliographystyle{plain}
\bibliography{references}
\end{document}